\documentclass[a4paper,12pt]{amsart}

\usepackage{a4wide}
\usepackage[english]{babel}
\usepackage{color}
\usepackage{graphicx}

\usepackage{amsthm,amssymb}

\usepackage[alphabetic]{amsrefs}

\usepackage{etoolbox}
\apptocmd{\sloppy}{\hbadness 10000\relax}{}{}

\newtheorem{theorem}{Theorem}
\newtheorem{corollary}[theorem]{Corollary}

\newtheorem{lemma}[theorem]{Lemma}
\newtheorem{proposition}[theorem]{Proposition}
\theoremstyle{definition}
\newtheorem{definition}[theorem]{Definition}
\theoremstyle{remark}
\newtheorem{remark}[theorem]{Remark}
\newtheorem{example}[theorem]{Example}
\newtheorem*{questions*}{Questions}

\newcommand{\NN}{\mathbb{N}}
\newcommand{\RR}{\mathbb{R}}
\newcommand{\CC}{\mathbb{C}}
\newcommand{\cont}{\mathcal{C}}
\newcommand{\hol}{\mathcal{O}}

\newcommand{\id}{\mathop{\mathrm{id}}}
\newcommand{\aut}{\mathop{\mathrm{Aut}}}

\newcommand{\slgrp}{{\mathrm{SL}}}
\newcommand{\slalg}{{\mathfrak{sl}}}

\author{Rafael B. Andrist}
\title[Integrable generators of Lie algebras]{Integrable generators of Lie algebras of vector fields on $\mathrm{SL}_2(\mathbb{C})$ and on $xy = z^2$}
\address{Rafael B. Andrist \\ Department of Mathematics \\
American University of Beirut \\
Beirut, Lebanon}
\address{Current Address: Faculty of Mathematics and Physics \\
University of Ljubljana \\
Ljubljana, Slovenia}

\thanks{The author was partially supported by the University Research Board (grant number: 104107) at the American University of Beirut (AUB) and by grant N1-0237 from ARRS, Republic of Slovenia}

\begin{document}

\begin{abstract}
For the special linear group $\mathrm{SL}_2(\mathbb{C})$ and for the singular quadratic Danielewski surface $x y = z^2$ we give explicitly a finite number of complete polynomial vector fields that generate the Lie algebra of all polynomial vector fields on them. Moreover, we give three unipotent one-parameter subgroups that generate a subgroup of algebraic automorphisms acting infinitely transitively on $x y = z^2$.
\end{abstract}

\keywords{density property, finitely generated Lie algebra, completely integrable vector fields, Andersen--Lempert theory, infinitely transitive}

\subjclass{32M17, 32E30, 32M25, 32Q56, 14R10}

\maketitle

\section{Introduction}

Several notions have been introduced to quantify that the automorphism group of Stein/affine manifold is ``large''. Around 2000, Varolin coined the notion of the density property (see Section \ref{secbackground} for details). One of the implications of the density property is the existence of finitely many complete holomorphic vector fields that span the tangent space $T_x X$ in every point $x \in X$, which implies transitivity of the group action. In fact, one can show that the group of holomorphic automorphisms acts $m$-transitively for any $m \in \NN$ which is called \emph{infinite transitivity}. The notion of infinite transitivity was introduced in 1999 by Kaliman and Zaidenberg \cite{MR1669174} in the algebraic category. We extend the definition to singular spaces:

\begin{definition}
Let $X$ be a complex variety and let $G$ be group acting on $X$ through (algebraic or holomorphic) automorphisms, then we call the action of $G$ \emph{infinitely transitive} if $G$ acts on the regular part $X_{\mathrm{reg}}$ $m$-transitively for any $m \in \NN$.
\end{definition}

In the algebraic category, the situation is slightly different; in particular, flows of complete algebraic vector fields need not be algebraic. Hence, the notion of the density property does not help in the study of algebraic automorphisms. Arzhantsev et al.\ introduced the notion of flexibility in 2013:
\begin{definition}{\cite{AFKKZ}}
A regular point $x \in X_{\mathrm{reg}}$ is called \emph{flexible} if the tangent space $T_x X$ is spanned by the tangent vectors to the orbits $Hx$ of one-parameter unipotent subgroups $H \subseteq \mathop{\mathrm{Aut}}(X)$.
A complex variety $X$ is called \emph{flexible} if every regular point $x \in X_{\mathrm{reg}}$ is flexible.
\end{definition}
One of their main results \cite{AFKKZ}*{Theorem 0.1} implies that the group of algebraic automorphisms of a flexible variety acts infinitely transitively.

The question whether one can find finitely many one-parameter unipotent subgroups that generate a subgroup of automorphisms acting infinitely transitively was first studied by Arzhantsev, Kuyumzhiyan and Zaidenberg \cite{AKZ2019} for toric varieties. For $\CC^n, n \geq 2,$ they showed that three one-parameter unipotent subgroups are sufficient. This was shown independently by the author using a different approach \cite{Andrist2019}; in the context of the volume density property, it seems natural instead to find finitely many complete vector fields with algebraic flows that generate the Lie algebra of all volume preserving algebraic vector fields. Similarly, one can consider the question of finding finitely many complete vector fields that generate the Lie algebra of all algebraic vector fields. In this case, the flows won't necessarily be algebraic anymore.

In this article, we focus on the question of generating the Lie algebra of all polynomial vector fields on the affine varieties $\slgrp_2(\CC)$ and $x y = z^2$. The corresponding one-parameter subgroups then generate a group that acts infinitely transitively. In Section \ref{secbackground} we provide the necessary tools and the theoretical background.

In Section \ref{secSL2} we consider the special linear group $\slgrp_2(\CC)$. Theorem \ref{thmSL2} shows there are four explicitly given vector fields that generate the Lie algebra of all polynomial vector fields on $\slgrp_2(\CC)$.

In Section \ref{secquadraticunipotent} we consider the singular quadratic Danielewski surface $\{(x,y,z) \in \CC^3 \,:\, x y = z^2\}$ which is a normal surface and also a toric variety; it has one isolated singularity in the origin. Theorem \ref{thmsingular} shows that there are four explicitly given vector fields that generate a certain Lie sub-algebra which is smaller than the Lie algebra of polynomial vector fields, but its flows nonetheless approximate holomorphic automorphisms that are isotopic to the identity, and can be used to achieve infinite transitivity.

In Section \ref{secquadraticunipotent} we revisit the same surface again, this time focusing on algebraic automorphisms, but no longer on generating the whole polynomial Lie algebra of vector fields. Our motivation is that \cite{AKZ2019}*{Theorem 5.20} only applies to toric varieties that are smooth in codimension $2$, which does not cover the case of the quadric $x \cdot y = z^2$. In Theorem \ref{thmalgsingapprox} we give three unipotent one-parameter subgroups that generate a subgroup of the algebraic automorphisms acting infinitely transitively on $x \cdot y = z^2$.

As a side note, we remark that the latter proof involves some tools from analysis such as the implicit function theorem to prove a purely algebraic result.

\begin{questions*}
Can we find finitely many complete polynomial vector fields that generate the Lie algebra of all polynomial vector fields on
\begin{enumerate}
\item\label{q1} smooth Danielewski surfaces $\{ (x,y,z) \in \CC^3 \,:\, x y = p(z) \}$ where $p \colon \CC \to \CC$ is a polynomial with simple roots,
\item\label{q2} singular quadrics $\{ (z_1, z_2, \dots, z_n) \in \CC^n \,:\, z_1^2 + z_2^2 + \dots + z_n^2 = 0 \}$, and
\item other simple Lie groups than $\slgrp_2(\CC)$?
\end{enumerate}
\end{questions*}
Note that the defining equation in Question (\ref{q2})  is equivalent to $x y = z^2$ by a linear change of coordinates if $n=3$.

\clearpage

\section{Background and Tools}
\label{secbackground}

\begin{definition}
Let $X$ be a complex variety and let $V$ be a holomorphic vector field on $X$. We call $V$ \emph{complete} or \emph{$\CC$-complete} if its flow map exists for all times $t \in \CC$. We call $V$ \emph{$\RR$-complete} if its flow map exists for all times $t \in \RR$. 
\end{definition}
Since the flow satisfies the semi-group property, any time-$t$ map of a $\RR$- or $\CC$-complete vector field is a holomorphic automorphism.

\smallskip

The density property for complex manifolds was introduced and studied by Varolin in \cites{MR1785520,MR1829353} around 2000:

\begin{definition} [\cite{MR1829353}] \hfill
\begin{enumerate}
\item Let $X$ be a Stein manifold. We say that $X$ has the \emph{density property} if the Lie algebra generated by the complete holomorphic vector fields on $X$ is dense (in the compact-open topology) is the Lie algebra of all holomorphic vector fields on $X$.
\item Let $X$ be an affine manifold. We say that $X$ has the \emph{algebraic density property}, if the Lie algebra generated by the complete algebraic vector fields on $X$ coincides with the Lie algebra of all algebraic vector fields on $X$.
\end{enumerate}
\end{definition}

By a standard application of Cartan's Theorem B and Cartan--Serre's Theorem A, the algebraic density property implies the density property (see e.g.\ \cite{MR3320241}*{Proposition 6.2} which also covers the singular case). Since flows of algebraic vector fields don't need to be algebraic, there is no direct advantage in proving the algebraic density property over the density property. But polynomial vector fields are much more amenable to algebraic manipulations, and thus the algebraic density property is usually easier to prove directly and can be used a useful tool for establishing the density property.

\begin{example}
Examples of Stein manifolds with the density property include $\CC^n,  n \geq 2$ which are a special case of affine homogeneous spaces of linear algebraic groups. The connected components of these homogeneous spaces enjoy the density property except for $\CC$ and $(\CC^\ast)^n$ \cites{MR2718937, MR3623226}. Whether or not $(\CC^\ast)^n$ has the density property is not known.  
Other classes of affine manifolds with the density property include smooth Danielewski surfaces $\{ (x,y,z) \in \CC^3 \,:\, xy = p(z) \}$ where $p$ is a polynomial with simple zeroes \cite{MR2350038}. Moreover, the Koras--Russel cubic threefold \cite{MR3513546}, Calogero--Moser spaces \cite{MR4305975} and a large class of Gizatullin surfaces \cites{MR3833804, MR3717940} also enjoy the density property. For details and a comprehensive list we refer the reader to the recent survey by Forstneri\v{c} and Kutzschebauch \cite{MR4440754}.
\end{example}

Let $X$ be a complex manifold of complex dimension $n$. We call a complex differential form of bi-degree $(n,0)$ on $X$ a \emph{volume form} if it is nowhere degenerate.

Let $X$ be a complex variety. We denote its \emph{group of holomorphic automorphisms} by $\aut(X)$. If $X$ is smooth and if there exists a volume form $\omega$ on $X$, we denote the \emph{group of $\omega$-preserving holomorphic automorphisms} by $\aut_{\omega}(X)$.

\pagebreak

\begin{definition}\hfill
\begin{enumerate}
\item Let $X$ be a Stein manifold with a holomorphic volume form $\omega$. We say that $(X, \omega)$ has the \emph{volume density property} if the Lie algebra generated by the complete $\omega$-preserving holomorphic vector fields on $X$ is dense (in the compact-open topology) in the Lie algebra of all $\omega$-preserving holomorphic vector fields on $X$. \cite{MR1829353}
\item Let $X$ be an affine manifold with an algebraic volume form $\omega$. We say that $(X, \omega)$ has the \emph{algebraic volume density property} if the Lie algebra generated by the complete $\omega$-preserving algebraic vector fields on $X$ coincides with the Lie algebra of all $\omega$-preserving algebraic vector fields on $X$. \cite{MR2660454}
\end{enumerate}
\end{definition}

Again, the algebraic volume density property implies the volume density property; however, the proof is not straightforward and can be found in \cite{MR2660454} by Kaliman and Kutzschebauch.

The main result for manifolds with density property is the following theorem which was first stated for star-shaped domains of $\CC^n$ by Anders\'en and Lempert in 1992, then generalized to Runge domains by Forstneri\v{c} and Rosay in 1993 and finally extended to manifolds with the density property by Varolin:
\begin{theorem}\cites{MR1185588, MR1213106, MR1296357, MR1829353}
\label{thmAL}
Let $X$ be a Stein manifold with the density property or $(X,\omega)$ be a Stein manifold with the volume density property, respectively. Let $\Omega \subseteq X$ be an open subset and $\varphi \colon [0,1] \times \Omega \to X$ be a $\cont^1$-smooth map such that
\begin{enumerate}
\item $\varphi_0 \colon \Omega \to X$ is the natural embedding,
\item $\varphi_t \colon \Omega \to X$ is holomorphic and injective for every $t \in [0,1]$ and, respectively, $\omega$-preserving, and
\item $\varphi_t(\Omega)$ is a Runge subset of $X$ for every $t \in [0,1]$.
\end{enumerate}
Then for every $\varepsilon > 0$ and for every compact $K \subset \Omega$ there exists a continuous family $\Phi \colon [0, 1] \to \aut(X)$ or (respectively) $\Phi \colon [0, 1] \to \aut_\omega(X)$, such that $\Phi_0 = \id_X$ and $\| \varphi_t - \Phi_t \|_K < \varepsilon$ for every $t \in [0,1]$.

\smallskip
Moreover, these each of the automorphisms $\Phi_t$ can be chosen to be compositions of flows of generators of a dense Lie subalgebra in the Lie algebra of all holomorphic vector fields on $X$.
\end{theorem}

One of the two main ingredients in the proof of Theorem \ref{thmAL} is the following proposition which has been found by Varolin \cite{MR1829353}, but is stated best as a stand-alone result in the textbook of Forstneric \cite{Forstneric-book}.

\begin{proposition}\cite{Forstneric-book}*{Corollary 4.8.4}
\label{propflowapprox}
Let $V_1, \dots, V_m$ be $\mathbb{R}$-complete holomorphic vector fields on a complex manifold $X$. Denote by $\mathfrak{g}$ the Lie subalgebra generated by the vector fields $\{V_1, \dots, V_m\}$ and let $V \in \mathfrak{g}$. Assume that $K$ is a compact set in $X$ and $t_0 > 0$ is such that the flow $\varphi_t(x)$ of $V$ exists for every $x \in K$ and for all $t \in [0, t_0]$. Then $\varphi_{t_0}$ is a uniform limit on $K$ of a sequence of compositions of time-forward maps of the vector fields $V_1, \dots, V_m$.
\end{proposition}

For the proof of Theorem \ref{thmalgsingapprox} where we can't make use of the density property, we will need to use Proposition \ref{propflowapprox} directly.

\bigskip

As one of many standard applications of Theorem \ref{thmAL} we obtain the following. It is implicit in the paper of Varolin \cite{MR1785520}, but can also be found with a detailed proof in \cite{Andrist2019}*{Lemma 7 and Corollary 8}.

\begin{proposition}
\label{propinftrans}
Let $X$ be a Stein manifold with the density property resp.\ $(X,\omega)$ be a Stein manifold with the volume density property with $\dim_\CC X \geq 2$. Let $\mathfrak{g}$ be a Lie algebra that is dense in the Lie algebra of all holomorphic vector fields on $X$ resp.\ in the Lie algebra of all $\omega$-preserving holomorphic vector fields on $X$. Then the group of holomorphic automorphisms generated by the flows of completely integrable generators of $\mathfrak{g}$ acts infinitely transitively on $X$.
\end{proposition}

\bigskip

The following lemma can be found in the proof of \cite{MR2385667}*{Corollary 2.2} by Kaliman and Kutzschebauch. Its proof is a straightforward calculation.

\begin{lemma}[Kaliman--Kutzschebauch formula]
Let $\Gamma$ and $\Delta$ be holomorphic vector fields and $f,g,h$ be holomorphic functions on a complex space. Then the following holds:
\begin{equation}
\label{KKformula}
\tag{KK}
[h \cdot f \cdot \Gamma, \; g \cdot \Delta] - [f \cdot \Gamma, \; h \cdot g \cdot \Delta]
=  - g f \Delta(h) \cdot \Gamma - f g \Gamma(h) \cdot \Delta 
\end{equation}
\end{lemma}

The power of this formula lies in the observation that if all the vector fields on the l.h.s.\ are in a certain Lie algebra and if in addition $\Gamma(h) = 0$, then we found a new vector field on the r.h.s.\ that is a multiple of $\Gamma$ and lies the same Lie algebra.

\bigskip

The notion of the density property was extended to singular varities by Kutzschebauch, Liendo and Leuenberger \cite{MR3320241}. Following them, we introduce these notations:
Let $X$ be a normal reduced Stein space and let $X_{\mathrm{sing}}$ be its singular locus. Let $A \subseteq X$ be a closed analytic subvariety containing $X_{\mathrm{sing}}$ and let $I_A \subseteq \hol(X)$ be the vanishing ideal of $A$.
Let $\mathrm{VF}_{\mathrm{hol}}(X,A)$ be the $\hol(X)$-module of vector fields vanishing in $A$. Let $\mathrm{Lie}_{\mathrm{hol}}(X,A)$ be the Lie algebra generated by all the complete vector fields in $\mathrm{VF}_{\mathrm{hol}}(X,A)$.

\begin{definition}\cite{MR3320241}*{Definition 6.1}
Let $X$ be a normal reduced Stein space and let $A \subset X$ be a closed subvariety. We say that $X$ has the \emph{(strong) density property relative to $A$} if the Lie algebra generated by the complete holomorphic vector fields vanishing in $A$ is dense in the Lie algebra of all holomorphic vector fields vanishing on $A$.
Furthermore, we say that $X$ has the \emph{weak density property relative to $A$} if there exists $\ell \geq 0$ such that $\mathcal{I}_A^\ell \cdot \mathrm{VF}_{\mathrm{hol}}(X, A) \subseteq \overline{\mathrm{Lie}_{\mathrm{hol}}(X, A)}$.
\end{definition}

\begin{theorem}\cite{MR3320241}*{Theorem 6.3}
Let $X$ be a normal reduced Stein space and let $A \subset X$ be a closed analytic subvariety that contains the singularity locus of $X$. 
Assume that $X$ has the weak relative density property with respect to $A$ for some $\ell \geq 0$.
Let $\Omega \subseteq X$ be an open subset and $\varphi \colon [0,1] \times \Omega \to X$ be a $\cont^1$-smooth map such that
\begin{enumerate}
\item $\varphi_0 \colon \Omega \to X$ is the natural embedding,
\item $\varphi_t \colon \Omega \to X$ is holomorphic and injective for every $t \in [0,1]$,
\item $\varphi_t(\Omega)$ is a Runge subset of $X$ for every $t \in [0,1]$, and
\item $\varphi_t$ fixes $A$ up to order $\ell$ where $\ell$ is such that \\
$\mathcal{I}_A^\ell \cdot \mathrm{VF}_{\mathrm{hol}}(X, A) \subseteq \overline{\mathrm{Lie}_{\mathrm{hol}}(X, A)}$.
\end{enumerate}
Then for every $\varepsilon > 0$ and for every compact $K \subset \Omega$ there exists a continuous family $\Phi \colon [0, 1] \to \aut(X)$, fixing $A$ pointwise, such that
$\Phi_0 = \id_X$ and $\| \varphi_t - \Phi_t \|_K < \varepsilon$ for every $t \in [0,1]$.

\smallskip
Moreover, these automorphisms can be chosen to be compositions of flows of generators of a dense Lie subalgebra in the Lie algebra of all holomorphic vector fields on $X$, see Varolin \cite{MR1829353}.
\end{theorem}

\section{The Special Linear Group $\slgrp_2(\CC)$}
\label{secSL2}

Throughout this section, we will use the coordinates $a,b,c,d$ for $\slgrp_2(\CC)$ in the following way:
\[
\slgrp_2(\CC) = \left\{
\begin{pmatrix} a & b \\ c & d \end{pmatrix} \,:\,
 a,b,c,d \in \CC, \;  ad - bc = 1
 \right\}
\]

The vector fields corresponding to left-multiplication by $\begin{pmatrix} 1 & t \\ 0 & 1 \end{pmatrix}$ and $\begin{pmatrix} 1 & 0 \\ t & 1 \end{pmatrix}$, respectively:
\begin{align*}
V &= c \frac{\partial}{\partial a} + d \frac{\partial}{\partial b}\\
W &= a \frac{\partial}{\partial c} + b \frac{\partial}{\partial d}\\
H &= c \frac{\partial}{\partial c} + d \frac{\partial}{\partial d} - a \frac{\partial}{\partial a} - b \frac{\partial}{\partial b} = [V,W]
\end{align*}
We have the commutation relations $[H,V] = 2V$ and $[H,W] = -2W$.

Similar to the case of the singular quadratic Danielewski surface, we obtain the following:
\begin{align*}
[V, bW] &= d W + b H \\
[V, d W + b H] &= 2 d H - 2 b V \\
[V, 2 d H - 2 b V] &=  - 6 d V \\
[V, aW] &= c W + a H \\
[V, c W + a H] &= 2 c H - 2 a V \\
[V, 2 c H - 2 a V] &=  - 6 c V \\
[W, d V] &= b V - d H \\
[W, b V - d H] &= -2 b H - 2 d W \\
[W, -2 b H - 2 d W] &= -6 b W \\
[[V, b^k W], b W] &= -(k+3) b^{k+1} W
\end{align*}
Using the above, we find the following lemma.

\begin{lemma}
The Lie algebra generated by $V, W, (b+c)W, dW$ contains all linear combinations of $V, W, H$ with polynomial coefficients of degree one.
\end{lemma}
\begin{proof}
\begin{align*}
[W, (b+c)W]          &= aW \\
[V,[V,[V,aW]]]       &= -6cV \\
[V,[V,[V, (b+c)W]    &= -6dV \\
[W,dW]               &= bW 
\end{align*}
From here, we now obtain also $cW = (b+c)W - bW$. Next,
\begin{align*}
[V, cW] &= cH \\
[W, cH] &= aH + 2cW
\end{align*}
yields $aH$ and $cH$.
We get the other terms in front of $H$ by proceeding symmetrically with $dW$:
\begin{align*}
[V, dW] &= d H \\
[W, dH] &= b H + 2 d W
\end{align*}
The missing term in front of $V$ then follow from:
\begin{align*}
[V, aH] &= c H - 2 a V \\
[V, bH] &= d H - 2 b V \qedhere
\end{align*}

\end{proof}

The following higher powers are easily obtained:
\begin{align*}
[[V, b^k W], b W] &= -(k+3) b^{k+1} W \\
[[V, a^k W], a W] &= -(k+3) a^{k+1} W \\
[[W, c^k V], c V] &= -(k+3) c^{k+1} V \\
[[W, d^k V], d V] &= -(k+3) d^{k+1} V
\end{align*}

Next we obtain $a^k H, b^k H, c^k H, d^k H$ using Equation \eqref{KKformula}:
\begin{align*}
[c W, c^k V]                  &= k a c^k V - c^{k+1} H  \\
[c W, c^k V] - [W, c^{k+1} V] &= -a c^k V 
\end{align*}
By taking linear combinations, we obtain $c^{k+1} H$. Analogously, we treat $a^{k+1} H, b^{k+1} H$ and $d^{k+1} H$.

Recall that 
\[
H = c \frac{\partial}{\partial c} + d \frac{\partial}{\partial d} - a \frac{\partial}{\partial a} - b \frac{\partial}{\partial b}
\]
It is now easy to see that Lie brackets of the form
\[
[ c^m H, a^k H] = (k + m) \cdot a^k c^m H
\]
will give us the following vector fields:
\[
a^k c^m H, \; a^k d^n H, \; b^\ell c^m H, \; b^\ell d^n H
\]
for any $k, \ell, m, n \in \NN_0$.
Next, 
\[
[a^k c^m H, b^\ell d^n H] = (-k+m-\ell+n) \cdot a^k b^\ell c^m d^n H
\]
yields all monomials in front of $H$ except those with $m+n = k+\ell$.  
The missing terms can be obtained e.g.\ using again Equation \eqref{KKformula}:
\[
[a \cdot a^k c^m H, b^\ell d^n H] - [a^k c^m H, a \cdot b^\ell d^n H] = 2 a^k b^\ell c^\ell d^n H \quad \text{ for } -\ell+n-1 \neq 0
\]
and
\[
[b \cdot a^k c^m H, b^\ell d^n H] - [a^k c^m H, b \cdot b^\ell d^n H] = 2 a^k b^\ell c^\ell d^n H \quad \text{ for } -k+m-1 \neq 0
\]
We can then transfer these powers to the terms with $V$ and $W$.
If both $-\ell+n-1 = 0$ and $-k+m-1 = 0$, then $m+n \neq k+\ell$, hence we obtained all monomial coefficients of $H$. We can transfer these terms now to $V$ and $W$ by
\begin{align*}
[V, a^k b^\ell c^\ell d^n H] = V(a^k b^\ell c^\ell d^n) H + 2 a^k b^\ell c^\ell d^n V
\end{align*}
and subtracting $V(a^k b^\ell c^\ell d^n) H$, since all such terms are known. Similarly, we proceed for $W$ and finally, we obtain all polynomial coeffients in front of $V$, $W$ and $H$. Since these three vector fields span $\slalg_2(\CC)$ as a vector space, the following theorem is now a standard application of Cartan--Serre's Theorem B.

\begin{theorem}
\label{thmSL2}
The four complete vector fields $V, W, (b+c)W, dW$ generate the Lie algebra of all polynomial vector fields on $\slgrp_2(\CC)$.
\end{theorem}

\begin{corollary}
The group of holomorphic automorphisms generated by the flows of $V$, $W$, $(b+c) W$ and $d H$ acts infinitely transitively on $\slgrp_2(\CC)$.
\end{corollary}

\begin{proof}
Theorem \ref{thmSL2} implies the (algebraic) density property for $\slgrp_2(\CC)$. This corollary then follows from Proposition \ref{propinftrans}.
\end{proof}

\section{Quadratic Singular Danielewski surface}
\label{secquadratic}

\begin{definition}
We consider the following singular quadratic Danielewski surface
\[
D = \{ (x,y,z) \in \CC^3 \,:\, x y = z^2 \}
\]
and the following complete vector fields
\begin{align*}
\Theta &= 2 z \frac{\partial}{\partial x} + y \frac{\partial}{\partial z} \\
\Xi    &= 2 z \frac{\partial}{\partial y} + x \frac{\partial}{\partial z} \\
H &= 2 y \frac{\partial}{\partial y} - 2 x \frac{\partial}{\partial x} = [\Theta, \Xi]
\end{align*}
\end{definition}
Note that $[H,\Theta] = 2 \Theta$ and $[H,\Xi] = -2 \Xi$, i.e.\ $\Theta$ and $\Xi$ form a $\slalg_2$-pair. However, the singular surface $D$ is not a homogeneous space of $\slgrp_2(\CC)$.

We will need the following equations where we act with $\Theta$ and $\Xi$, respectively, from the left.
\begin{align}
[\Theta, x \Xi] &= 2 z \Xi + x H \label{th-act1} \\
[\Theta, [\Theta, x \Xi]] &= 2 y \Xi + 4 z H - 2 x \Theta \label{th-act2} \\
[\Theta, [\Theta, [\Theta, x \Xi]]] &= 6 y H - 12 z \Theta \label{th-act3} \\
[\Theta, [\Theta, [\Theta, [\Theta, x \Xi]]]] &= - 24 y \Theta \label{th-act4} \\
[\Xi, y \Theta] &= 2 z \Theta - y H \label{xi-act1} \\
[\Xi, [\Xi, y \Theta]] &= 2 x \Theta - 4 z H - 2 y \Xi \label{xi-act2} \\
[\Xi, [\Xi, [\Xi, y \Theta]]] &= -6 x H - 12 z \Xi \label{xi-act3} \\
[\Xi, [\Xi, [\Xi, [\Xi, y \Theta]]]] &= - 24 x \Xi \label{xi-act4}
\end{align}
Note that if not sending it directly to zero, none of $\Theta$, $\Xi$ and $H$ can reduce the polynomial degree. We therefore need $\Theta$ and $\Xi$ (or linear combination thereof) for sure, since they can't be produced otherwise. Also note that the r.h.s.\ of the equations \eqref{th-act2} and \eqref{xi-act2} only differ by a sign.
The above computation shows that $\Theta$, $\Xi$ and $x \Xi$ can generate $y \Theta$. Similarly, $\Theta$, $\Xi$ and $y \Theta$ can generate $x \Xi$.

\begin{lemma}
\label{lem-singular-algebraic}
The Lie algebra generated by $\Xi, \Theta, x \Xi$ contains
\[
H, \; x^k \Xi, \; y^k \Theta \quad \mbox{for all} \; k \in \NN
\]
\end{lemma}
\begin{proof}
By induction in $k \in \NN$ we obtain all terms of the form $x^k \Xi$:
\[
[ [\Theta, x^k \Xi], x \Xi] = [2k x^{k-1} z \Xi + x^k H, x \Xi] = - (2k+4) \cdot x^{k+1} \Xi
\]
From equations \eqref{th-act1}--\eqref{th-act4} we obtain $y \Theta$ and proceed with the calculation analogous to the above, where we reverse the roles of $\Theta$ and $\Xi$ as well as $x$ and $y$, respectively.
\end{proof}

Note that the flows of $\Xi, \Theta, x \Xi$ are all algebraic. 
If we also allow holomorphic flows of complete algebraic vector fields, then we can generate a larger Lie algebra containing the above, if we take $\Xi, \Theta, z \Xi$ as generators:

\begin{lemma}
\label{lem-singular-alllinear}
The Lie algebra generated by $\Xi, \Theta, z H, z \Xi$ contains all the vector fields with all polynomial coefficients of degree one in front of $\Xi, \Theta, H$.
\end{lemma}
\begin{proof}
We proceed step by step as follows:
\begin{align}
[\Theta, z \Xi] &= y \Xi + z H \label{th-nonvol1} \\
[\Theta, y \Xi + z H] &= 2 y H - 2 z \Theta \label{th-nonvol2} \\
[\Theta, 2 y H - 2 z \Theta] &= - 6 y \Theta
\end{align}
After obtaining $y \Theta$, we use the equations from \eqref{xi-act1} to \eqref{xi-act4} to obtain $x \Xi$. By taking linear combinations of \eqref{th-nonvol2} and \eqref{xi-act1}, we obtain $yH$ and $z\Theta$. Equation \eqref{th-act1} gives us also $xH$.
Finally, using $zH$ and equation \eqref{th-nonvol1} we obtain $y \Xi$. Analogous to \eqref{th-nonvol1} we have 
\[
[\Xi, z \Theta] = x \Theta - z H
\]
and thus obtain $x \Theta$.
\end{proof}



\begin{theorem}
\label{thmsingular}
The four complete vector fields $\Xi, \Theta, z \Xi, z H$ generate the Lie algebra
\[\{f \cdot \Xi + g \cdot \Theta + h \cdot H \,:\, f,g,h \in \CC[M] \}\]
on the singular quadratic Danielewski surface $M := \{(x,y,z) \in \CC^3 \,:\, x y - z^2 = 0 \}$.
\end{theorem}
\begin{proof} \hfill
\begin{enumerate}
\item Using Lemma \ref{lem-singular-algebraic} and Lemma \ref{lem-singular-alllinear} we obtain the following (actually complete) vector fields for all $k \in \NN_0$:
\begin{align*}
[y \Xi, x^k \Xi] &= -2z x^k \Xi \\
[x \Theta, y^k \Theta] &= -2z y^k \Theta 
\end{align*}
\item Next, we obtain all vector fields of the form $z^k \Xi$. We proceed by induction in $k$. However, the induction step breaks down when passing from $z^2 \Xi$ to $z^3 \Xi$.
\begin{align*}
[z^k \Xi, y \Xi] &= (2 - k) z^{k+1} \Xi \\
[z \Xi, y H] - [\Theta, [z \Xi, y \Xi]] &= z^2 H \\
[\Theta, z^2 \Xi] &= 2 x y \Xi - z^2 H \\
[z \Xi, y z \Xi] &= 2 z^3 \Xi
\end{align*}
\item
Similarly, we can obtain all vector fields of the form $z^k \Theta$.
\item
\[
(k+1)[z^k \Xi, y H] - 2[\Theta, z^{k+1} \Xi] = 2k z^{k+1} H
\]
\item
Let $f(y), g(x), h(z)$ be any polynomial in $y, x, z$, respectively.
\\
We apply the Kaliman--Kutzschebauch formula \eqref{KKformula} twice:
\begin{align*}
[z f(y) \Theta, h(z) H]-[f(y) \Theta, z h(z) H] &= -f(y) h(z) y H \\
[z g(x) \Xi, -f(y) h(z) y H]-[g(x) \Xi, -z \cdot f(y) h(z) y H] &=  x y \cdot g(x) f(y) h(z) H
\end{align*}
Similarly, we can also obtain
\begin{align*}
[z g(x) \Xi, h(z) H]-[g(x) \Xi, z h(z) H] &= -g(x) h(z) x H
\end{align*}
By taking linear combinations, we obtain all polynomial coefficients in front of $H$.

Similarly, we obtain a result for $\Theta$, while the calculation for $\Xi$ would be completely analogous to $\Theta$:
\begin{align*}
[y f(y) \Theta, g(x) h(z) H]-[f(y) \Theta, y g(x) h(z) H] &= -2y f(y) g(x) h(z) \Theta \\
[x f(y) \Theta, g(x) h(z) H]-[f(y) \Theta, x g(x) h(z) H] &= 2x f(y) g(x) h(z) \Theta - 2z f(y) g(x) h(z) H
\end{align*}
Thus, we have obtained all polynomial coefficients in front of $\Xi$, $\Theta$ and $H$. \qedhere
\end{enumerate}
\end{proof}

\begin{lemma}

The vector fields $\Xi, \Theta, H$ together span the tangent space of $\{(x,y,z) \in \CC^3 \,:\, x y = z^2 \}$ in each point except the origin.
\end{lemma}

\begin{proof}
Let us recall the definition of these vector fields:
\begin{align*}
\Theta &= 2 z \frac{\partial}{\partial x} + y \frac{\partial}{\partial z} \\
\Xi    &= 2 z \frac{\partial}{\partial y} + x \frac{\partial}{\partial z} \\
H &= 2 y \frac{\partial}{\partial y} - 2 x \frac{\partial}{\partial x} = [\Theta, \Xi]
\end{align*}
On each point of $\{ x \neq 0 \} \cap \{ y \neq 0 \}$ the vector fields $\Theta$ and $\Xi$ are linearly independent in $\CC^3$ and hence must span the $2$-dimensional tangent space.
On $\{ x = 0 \} \cap \{ y \neq 0 \}$ the vector fields $\Theta$ and $H$ span, and on $\{ y = 0 \} \cap \{ x \neq 0 \}$ the vector fields $\Xi$ and $H$ span, by the same argument.
\end{proof}

\begin{remark}
One might also decide to work with $\Xi$, $\Theta$ and some of their pullbacks instead: Pulling back $\Xi$ with the flow of $\Theta$ for fixed time $t$ -- and vice versa, we obtain new vector fields $\widetilde \Xi_t$, $\widetilde \Theta_t$. A direction calculation yields: 
\begin{align*}
\widetilde{\Xi}_t &= 2t (x - t z) \frac{\partial}{\partial x} + 2 (z - t y) \frac{\partial}{\partial y} + (x - t^2 y) \frac{\partial}{\partial z} = \Xi - t H - t^2 \Theta, \\
\widetilde{\Theta}_t &= 2 (z - t x) \frac{\partial}{\partial x} +  2t (y - t z)\frac{\partial}{\partial y} + (y - t^2 x) \frac{\partial}{\partial z}  = \Theta - t H - t^2 \Xi,
\end{align*}
which shows that these pullbacks are already in the span of $\Xi, \Theta$ and $H$.

\smallskip

This can also be compared to a result on algebraic ellipticity in the monograph of Alarc{\'o}n, Forstneri\v{c} and L{\'o}pez \cite{MR4237295}*{Proposition 1.15.3} where vector fields $V^1, V^2, V^3$ with polynomial flows are given explicitly in coordinates $(z_1, z_2, z_3) \in \CC^3$ on the quadric $z_1^2 + z_2^2 + z_3^2 = 0$; they are related by
\[
V^3 = \frac{1}{2} [V^1, V^2] - \frac{1}{4}[V^1 + V^2,[V^1, V^2]]
\]
By a global linear change of coordinates, we can map our $\frac{+i}{2}\Xi$ and $\frac{-i}{2}\Theta$ to those $V^1$ and $V^2$, respectively:
\[
(z_1 + i z_2, \, z_1 - i z_2, \, i z_3) = (x,y,z)
\]
\end{remark}

\begin{corollary}
\label{corsinginftrans}
The group of the holomorphic automorphisms generated by the flows of $\Xi$, $\Theta$, $z \Xi$ and $z H$ acts infinitely transitively on the regular locus of $\{ (x,y,z) \in \CC^3 \,:\, x y = z^2 \}$.
\end{corollary}

\begin{proof}
Theorem \ref{thmsingular} implies the weak algebraic relative density property for $\{ (x,y,z) \in \CC^3 \,:\, x y = z^2 \}$ using the flows of the generators above. This corollary then follows from the same proof as in Proposition \ref{propinftrans}, but we instead have to use relative version with respect to the singularity in the origin.
\end{proof}

\begin{remark}
The strong algebraic density property of $x y = z^2$ relative to the origin is known (but not with finitely many generators) due to \cite{MR3320241}*{Corollary 5.5} with $d=2, e=1$. 
%
%
\end{remark}

\section{Quadratic Singular Danielewski surface with unipotent subgroups}
\label{secquadraticunipotent}

In this section we discuss how to find finitely many unipotent one-parameter subgroups that generate a subgroup of the algebraic automorphisms and act infinitely transitively on $M := \{ (x,y,z) \in \CC^3 \,:\, x y = z^2 \}$. We first consider two obvious approaches that do not quite work.

\begin{remark}
Outside the singularity in the origin, $M$ is equipped with an algebraic volume form
\begin{equation}
\omega := \frac{dx \wedge dz}{x} = -\frac{dy \wedge dz}{y}
\end{equation}

Recall that by Lemma \ref{lem-singular-algebraic} the Lie algebra generated by $\Xi, \Theta, x\Xi$ contains
\[
\quad H, x^k \Xi, y^k \Theta \quad \mbox{for all} \; k \in \NN.
\]
Each of the vector fields $\Xi, \Theta, x\Xi$ is a locally nilpotent derivation that preserves the algebraic volume form $\omega$. However, it is not clear how other volume preserving vector fields of the form $z^k H$ could be obtained as Lie combinations or even be approximated. Thus, unlike in \cite{Andrist2019} any kind of transitivity result can't follow from an application of the volume density property.
\end{remark}

\begin{remark}
According to the result of Arzhantsev~et~al. \cite{AFKKZ}*{Theorem 2.5} a subgroup generated by algebraic one-parameter subgroups which is saturated and acts with an open orbit, is in fact acting infinitely transitively on that open orbit. 

All the polynomial shears (or, in the terminology of \cite{AFKKZ}: \emph{replicas}) of the locally nilpotent derivations $\Xi$ and $\Theta$ are of the form $f(x) \Xi$ and $g(y) \Theta$, respectively.
Let $G$ be the group generated by their flows. Then we can't expect that the conjugates of $f(x) \Xi$ and $g(y) \Theta$ by group elements in $G$ can be obtained by Lie combinations of $f(x) \Xi$ and $g(y) \Theta$. Hence, the saturation condition of \cite{AFKKZ}*{Theorem 2.5} is not satisfied.
\end{remark}
In the following, let $S := \{ x = 0 \} \cup \{ y = 0 \}$ be the set where $\Theta$ and $\Xi$ are not spanning the tangent space.
 
\begin{lemma}
\label{lemsingapprox}
Let $B_1, \dots, B_m \subset \CC^3$ be balls of radius $\varepsilon > 0$ centered in $p_1, \dots, p_m \in M \setminus S$ respectively, and let $v_0 \in T_{p_m}$. For sufficiently small $\varepsilon > 0$ and for any $\delta > 0$ there exist polynomials $f, g \in \CC[z]$ such that the following hold for the vector field $V := f(x) \Xi + g(y) \Theta$.
\begin{enumerate}
\item $\displaystyle \| V \|_{B_j} < \delta$ for $j = 1, \dots, m-1$
\item $\displaystyle \| V - v_0 \|_{B_m} < \delta$
\end{enumerate}
\end{lemma}

\begin{proof}
For small enough $\varepsilon > 0$, the projection of the union of $B_1, \dots, B_m \subset \CC^3$ to the $x$-axis and to the $y$-axis is Runge. By the Runge approximation theorem, we find holomorphic functions $\widetilde f, \widetilde g \colon \CC \to \CC$ such that $\widetilde f(x) \Xi + \widetilde g(y) \Theta$ satisfies the desired approximation with an estimate of $\delta/2$ instead of $\delta$. For point (2) observe that $\Xi$ and $\Theta$ are spanning the tangent space in $M \setminus S$. Finally, Taylor expand $\widetilde f$ and $\widetilde g$ inside a large enough disk $P \subset \CC$ that contains the projections of $B_1, \dots, B_m$, such that for their respective Taylor polynomials $f$ and $g$ we have that $\| f - \widetilde f \|_P < \delta/2$ and $\| g - \widetilde g \|_P < \delta/2$. Then $f$ and $g$ are the desired polynomials.
\end{proof}

\begin{proposition}
\label{propmainsing}
Let $p_1 = q_1, p_2 = q_2, \dots, p_{m-1} = q_{m-1}$ and $p_m \neq q_m$ be pairwise different points in $M \setminus S$. Then there exists an algebraic automorphism $F \colon M \to M$ that is a finite composition of time-$1$ maps of locally nilpotent derivations of the form $f(x) \Xi$ and $g(y) \Theta$ and such that $F(p_1) = q_1 = p_1, \dots, F(p_{m-1}) = q_{m-1} = p_{m-1}, F(p_m) = q_m \neq p_m$.
\end{proposition}

\begin{proof}
Since $M \setminus S$ is connected, we can choose a path $\gamma$ in $M \setminus S$ from $p_m$ to $q_m$ that avoids $p_1, p_2, \dots, p_{m-1}$. Choose $\varepsilon_0 > 0$ s.t. $\min\limits_{j=1, \dots, m-1} \inf\limits_t \| \gamma(t) - p_j \| > \varepsilon_0$ and $\min\limits_{j \neq k} \| p_j - p_k \| > 2 \varepsilon_0$.
Let $B_1, \dots, B_m$ be closed balls in $\CC^3$ of radius $\varepsilon \in (0, \varepsilon_0)$ around $p_1, \dots, p_m$, respectively. Set $K = \bigcup_{j=1}^m B_j$.

We now proceed with the same geometric idea as in the proof of \cite{Andrist2019}*{Lemma 7}, but without making use of the (volume) density property, but instead rely on Lemma \ref{lemsingapprox} above.

For each point $j \in \{1, \dots, m\}$ let $V^{j,x}$ and $V^{j,y}$ be the holomorphic vector fields defined on $K$ which vanish on $B_1, \dots, B_{j-1}, B_{j+1}, \dots, B_{m}$ and agree with the partial derivative $\displaystyle \frac{\partial}{\partial x}$ and $\displaystyle \frac{\partial}{\partial y}$, respectively, on $B_j$. 
For small enough time, the flows of $V^{j,x}$ and $V^{j,y}$ exist on $K$ and can be approximated arbitrarily well by finite compositions $F^{j,x}_t$ and $F^{j,y}_t$, respectively, of flows of $\Xi, \Theta, x \Xi$: Since Lemma \ref{lem-singular-algebraic} established that the Lie algebra generated by $\Xi, \Theta, x\Xi$ contains all the vector fields $f(x) \Xi + g(y) \Theta$ with polynomials $f$ and $g$, we can apply now Proposition \ref{propflowapprox} on the complex manifold $M \setminus S$ to the vector fields produced by the application of Lemma \ref{lemsingapprox} to each of the vector fields $V^{j,x}$ and $V^{j,y}$.

Consider the map $\Phi \colon (\CC^2)^{m} \to (M \setminus S)^{m}$ given by
\[
\begin{pmatrix}
 (t_{1,x}, t_{1,y}) \\
 \vdots \\
 (t_{m,x}, t_{m,y})
\end{pmatrix}
  \mapsto
\begin{pmatrix}
 F^{m,y}_{t_{m,y}} \circ F^{m,x}_{t_{m,x}} \circ \dots \circ F^{1,y}_{t_{1,y}} \circ F^{1,x}_{t_{1,x}}(p_1) \\
 \vdots \\
 F^{m,y}_{t_{m,y}} \circ F^{m,x}_{t_{m,x}} \circ \dots \circ F^{1,y}_{t_{1,y}} \circ F^{1,x}_{t_{1,x}}(p_{m})
\end{pmatrix}
\]
For a sufficiently close approximation, this map is submersive in $0 \in (\CC^2)^{m}$. Now by the implicit function theorem there exists a neighborhood $U_1 \times \dots \times U_{m}$ of $(p_1, \dots, p_{m}) \in (M \setminus S)^{m}$ and a neighborhood $V_1 \times \dots \times V_{m}$ of $(0, \dots, 0) \in (\CC^2)^{m}$ such that $\Phi \colon V_1 \times \dots \times V_{m} \to U_1 \times \dots \times U_{m}$ is a surjective (in fact, bijective) holomorphic map. In particular, for each $r \in U_{m} =: U$ we find $(t_{1,x}, t_{1,y}) \in V_1, \dots, (t_{m,x}, t_{m,y}) \in V_{m}$ s.t.\
\[
\begin{pmatrix}
 F^{m,y}_{t_{m,y}} \circ F^{m,x}_{t_{m,x}} \circ \dots \circ F^{1,y}_{t_{1,y}} \circ F^{1,x}_{t_{1,x}}(p_1) \\
 \vdots \\
 F^{m,y}_{t_{m,y}} \circ F^{m,x}_{t_{m,x}} \circ \dots \circ F^{1,y}_{t_{1,y}} \circ F^{1,x}_{t_{1,x}}(p_{m-1}) \\
 F^{m,y}_{t_{m,y}} \circ F^{m,x}_{t_{m,x}} \circ \dots \circ F^{1,y}_{t_{1,y}} \circ F^{1,x}_{t_{1,x}}(r) \\ 
\end{pmatrix}
=
\begin{pmatrix}
p_1 \\
 \vdots \\
p_{m-1} \\
y
\end{pmatrix}
\]
Note that the choice of the times depends holomorphically on $y$ without any further control, but the map $F := F^{m,y}_{t_{m,y}}  \circ F^{m,x}_{t_{m,x}} \circ \dots \circ F^{1,y}_{t_{1,y}} \circ F^{1,x}_{t_{1,x}} \colon M \setminus S \to M \setminus S$ is a finite composition of algebraic automorphisms.

For each point $r$ on the trace of the path $\gamma$ we can apply the above procedure. Since this trace is compact, we can cover it by finitely many balls of some radius $\varepsilon \in (0, \varepsilon_0)$ and small enough to apply the implicit function theorem as above. Composing these maps, we finally obtain an algebraic automorphism that moves $p_m$ to $q_m$, but keeps the other points $p_1, \dots, p_{m-1}$ fixed.
\end{proof}

\begin{theorem}
\label{thmalgsingapprox}
The subgroup of the algebraic automorphisms generated by the flows of the locally nilpotent derivations $\Xi$, $\Theta$ and $x \Xi$ acts infinitely transitively on the regular locus of $\{ (x,y,z) \in \CC^3 \,:\, x y = z^2 \}$.
\end{theorem}
\begin{proof}
Let $0, p_1, \dots, p_m \in M$ be pairwise distinct points and let $0, q_1, \dots, q_m \in M$ be another pairwise distinct points. Our goal is to find a map $F \colon M \to M$ such that $F(p_1) = q_1, \dots, F(p_m) = q_m$. 
After an algebraic change of coordinates $G \colon M \to M$, which we obtain by composing the flows for $\Theta$ and for $\Xi$ any generic choice of times, we may assume that $p_1, \dots, p_m, q_1 \dots, q_m \in M \setminus S$. The result now follows from an inductive application of Proposition \ref{propmainsing}, moving each point separately, while keeping the others fixed. Finally, we conjugate the obtained map by $G$.
\end{proof}

\begin{remark}
The proofs of Proposition \ref{propmainsing} and hence of Theorem \ref{thmalgsingapprox} work in general, provided one can prove an analog of Lemma \ref{lemsingapprox} which essentially requires the existence of some ``nice'' projections to subvarieties that are contained in the kernels of the locally nilpotent derivations that span the tangent spaces.
\end{remark}

\section*{Acknowledgements}

The author would like to Franc Forstneri\v{c} and Frank Kutzschebauch for helpful comments on this paper.

\end{document}